\documentclass[12pt]{amsart}
\makeatletter
\newcommand*{\rom}[1]{\expandafter\@slowromancap\romannumeral #1@}

\usepackage{amsmath,amsthm,amssymb,comment,fullpage}

\usepackage[T1]{fontenc}
\usepackage{latexsym}
\usepackage[dvips]{graphics}
\usepackage{epsfig}
\usepackage{yhmath}
\usepackage{amsmath,amsfonts,amsthm,amssymb,amscd}
\input amssym.def
\input amssym.tex
\usepackage{color}
\usepackage{bbold}
\usepackage{hyperref}
\usepackage{amsthm}
\usepackage{url}
\usepackage{dsfont}

\numberwithin{equation}{section}

\newtheorem{thm}{Theorem}[section]

\newtheorem{problem}[thm]{Problem}

\theoremstyle{plain}

\newtheorem{definition}[thm]{Definition}

\newtheorem{lemma}[thm]{Lemma}
\newtheorem{proposition}[thm]{Proposition}
\newtheorem{theorem}[thm]{Theorem}

\newtheorem*{notation*}{Notation}

\newtheorem{remark}[thm]{Remark}

\newcommand\be{\begin{equation}}
\newcommand\ee{\end{equation}}
\newcommand\bea{\begin{eqnarray}}
\newcommand\eea{\end{eqnarray}}
\newcommand\bi{\begin{itemize}}
\newcommand\ei{\end{itemize}}
\newcommand\ben{\begin{enumerate}[(a)]}
\newcommand\een{\end{enumerate}}
\newcommand\bc{\begin{center}}
\newcommand\ec{\end{center}}
\def\ba#1\ea{\begin{align*}#1\end{align*}}

\newcommand{\R}{\ensuremath{\mathbb{R}}}
\newcommand{\C}{\ensuremath{\mathbb{C}}}
\newcommand{\Z}{\ensuremath{\mathbb{Z}}}

\newcommand{\N}{\mathbb{N}}
\newcommand{\E}{\mathbb{E}}

\newcommand{\T}{\mathbb{T}}

\usepackage{enumerate}

\begin{document}
\title[Multicorr. sequences and large returns]{Properties of multicorrelation sequences and large returns under some ergodicity assumptions}



\author{Andreu Ferr\'e Moragues}
\address{Department of Mathematics, Ohio State University, Columbus, OH 43210, USA}
\email{ferremoragues.1@osu.edu}
\maketitle

\begin{abstract}
	We prove that given a measure preserving system $(X,\mathcal{B},\mu,T_1,\dots,T_d)$ with commuting, ergodic transformations $T_i$ such that $T_iT_j^{-1}$ are ergodic for all $i \neq j$, the multicorrelation sequence $a(n)=\int_X f_0 \cdot T_1^nf_1 \cdot \dotso \cdot T_d^n f_d \ d\mu$ can be decomposed as $a(n)=a_{\textrm{st}}(n)+a_{\textrm{er}}(n)$, where $a_{\textrm{st}}$ is a uniform limit of $d$-step nilsequences and $a_{\textrm{er}}$ is a nullsequence (that is, $\lim_{N-M \to \infty} \frac{1}{N-M} \sum_{n=M}^{N-1} |a_{\textrm{er}}|^2=0$). Under some additional ergodicity conditions on $T_1,\dots,T_d$ we also establish a similar decomposition for polynomial multicorrelation sequences of the form $a(n)=\int_X f_0 \cdot \prod_{i=1}^dT_i^{p_{i,1}(n)}f_1\cdot\dotso \cdot \prod_{i=1}^dT_i^{p_{i,k}(n)}f_k \ d\mu$, where each $p_{i,k}: \Z \rightarrow \Z$ is a polynomial map. We also show, for $d=2$, that if $T_1, T_2, T_1T_2^{-1}$ are invertible and ergodic, we have large triple intersections: for all $\varepsilon>0$ and all $A \in \mathcal{B}$, the set $\{n \in \Z : \mu(A \cap T_1^{-n}A \cap T_2^{-n}A)>\mu(A)^3-\varepsilon\}$ is syndetic. Moreover, we show that if $T_1, T_2, T_1T_2^{-1}$ are totally ergodic, and we denote by $p_n$ the $n$-th prime, the set $\{n \in \N : \mu(A \cap T_1^{-(p_n-1)}A \cap T_2^{-(p_n-1)}A)>\mu(A)^3-\varepsilon\}$ has positive lower density.
\end{abstract}
\section{Introduction}
In this paper we obtain, under some ergodicity assumptions on the involved transformations, about the multicorrelation sequence $I_{\vec{f},d}(n)$ which is defined as follows. Let $(X,\mathcal{B},\mu)$ be a probability space, let $T_1,\dots,T_d: X \rightarrow X$ be commuting, invertible, measure preserving transformations, and let $\vec{f}=(f_0,\dots,f_d) \in (L^{\infty}(\mu))^{(d+1)}$. Then we put
\begin{equation}\label{ivector} I_{\vec{f},d}(n):=\int_X f_0 \cdot T_1^{n}f_1 \cdot \dotso \cdot T_d^n f_d \ d\mu.
\end{equation}
When $d=1$, Herglotz-Bochner's theorem implies that the correlation sequence $I_{\vec{f},1}(n)$ is given by the Fourier coefficients of some finite complex measure $\sigma$ on $\T$ (see \cite{khintchine} and \cite{koopmanvonneumann}). Decomposing $\sigma$ into its atomic part, $\sigma_a$, and continuous part, $\sigma_c$, we get
\begin{equation}\label{decompherglotz}
	I_{\vec{f},1}(n)=\int_X f_0 \cdot T_1^nf_1 \ d\mu=\int_{\T} e^{2\pi i nx} \ d\sigma(x)=\int_{\T} e^{2\pi i n x} d\sigma_a(x)+\int_{\T} e^{2\pi i n x} d\sigma_c(x)=\psi(n)+\nu(n).\end{equation}
In formula \eqref{decompherglotz}, the sequence $\psi(n)$  is \emph{almost periodic}, which, equivalently, means that there exists a compact abelian group $G$, a continuous function $\phi: G \rightarrow \C$, and $a \in G$ such that $\psi(n)=\phi(a^n)$.
\\ \\
On the other hand, the sequence, $\nu(n)$ is a nullsequence in the sense of the following definition:
\begin{definition}\label{nullseq1} Let $a: \Z \rightarrow \C$ be a bounded sequence. We say that $a$ is a \emph{null-sequence} if
	\begin{equation}\label{besicovitch} \lim_{N-M \to \infty} \frac{1}{N-M} \sum_{n=M}^{N-1} |a(n)|^2=0.
	\end{equation}
\end{definition}
It is natural to inquire whether for general $d \in \N$ the multicorrelation sequence $I_{\vec{f},d}(n)$ can be represented as a sum of a ``generalized almost periodic sequence'' and a nullsequence. A very satisfactory answer to this question was obtained by Bergelson, Host and Kra in \cite{bhk} for the case $T_i=T^i$. Before formulating their result, we need to introduce some definitions pertaining to systems on nilmanifolds:
\begin{definition} Let $G$ be a Lie group. For a positive integer $d$, we denote by $G_d$ the $d$-th commutator subgroup of $G$.
	\begin{itemize}
		\item  The group $G$ is a $d$-\emph{step nilpotent} Lie group if $G_{d+1}=\{e\}$. 
		\item Let $\Gamma$ be a discrete cocompact subgroup of $G$. The quotient $X=G/\Gamma$ is a $d$-step nilmanifold if $G$ is a $d$-step nilpotent Lie group. 
		\item Let $X=G/\Gamma$ be a $d$-step nilmanifold and $a \in G$. We call the map $T: X \rightarrow X$ given by $T(g\Gamma)=(ag)\Gamma$ a niltranslation.
		\item Finally, let $X$ be a $d$-step nilmanifold, $f \in C(X)$, $T$ a niltranslation, and $x \in X$. The sequence $a(n)=f(T^nx)$ is called a $d$-\emph{step nilsequence}.
	\end{itemize}
\end{definition}
In \cite{bhk} it was proved that 
\begin{equation}\label{bhkdecomp1} \int_X f_0 \cdot Tf_1 \cdot \dotso \cdot T^df_d \ d\mu=a_{\textrm{st}}(n)+\nu(n),
\end{equation}
where $a_{\textrm{st}}(n)$, the ``structured part'', is a uniform limit of $d$-step nilsequences and $\nu(n)$ is a nullsequence (see Theorem 1.9 in \cite{bhk}).
\\ \\
In \cite{leibman1}, Leibman showed that a similar decomposition holds for multicorrelation sequences of the form $b(n)=\int_X f_0 \cdot T^{p_1(n)}f_1 \cdot \dotso \cdot T^{p_d(n)}f_d \ d\mu$, where $p_1,\dots,p_d \in \Z[n]$. He proved that $b(n)$ can be decomposed as a sum of a uniform limit of $r$-step nilsequences and a nullsequence. (The number $r$ depends on the family of polynomials $\{p_1,\dots,p_d\}$.)
\\ \\
For general commuting transformations, the following result was proved by Frantzikinakis in \cite{fra1}:
\begin{theorem}[Theorem 1.1 in \cite{fra1}]\label{fratheorem1}
	Let $d \in \N$. Let $(X,\mathcal{B},\mu)$ be a probability space and $T_1,\dots,T_d: X \rightarrow X$ be commuting measure preserving transformations of $X$. Let $f_0,\dots,f_d \in L^{\infty}(\mu)$ and $\varepsilon>0$. Then, the sequence $I_{\vec{f},d}(n)$ in \eqref{ivector} admits the following decomposition:
	\begin{equation}\label{fradecomp1} I_{\vec{f},d}(n)=a_{\textrm{st}}(n)+a_{\textrm{er}}(n),
	\end{equation}
	where $a_{\textrm{st}}(n)$ is a uniform limit of $d$-step nilsequences, and $a_{\textrm{er}}(n)$, the so-called error sequence, is a bounded sequence with $||a_{\textrm{er}}||_2:=\limsup_{N-M \to \infty} \frac{1}{N-M} \sum_{n=M}^{N-1} |a_{\textrm{er}}(n)|^2\leq \varepsilon$.
\end{theorem}
Juxtaposing equations \eqref{bhkdecomp1} and \eqref{fradecomp1}, it is natural to ask whether the $\varepsilon$ in Theorem \ref{fratheorem1} can be removed. Indeed, in \cite{hostfra}, Host and Frantzikinakis posed the following question:
\begin{problem}\label{questionintro}[Problem 1 in \cite{hostfra}] Let $(X,\mathcal{B},\mu)$ be a probability space, and let $T_1,T_2: X \rightarrow X$ be measure preserving transformations that commute with one another. Let $f_0,f_1,f_2 \in L^{\infty}(\mu)$. Is it true that the multicorrelation sequence $a: \N \rightarrow \C$ defined by
	\[ a(n):=\int_X f_0 \cdot T_1^nf_1 \cdot T_2^n f_2 \ d\mu, \quad n \in \N\]
	can be decomposed as $a=a_{\textrm{st}}+a_{\textrm{er}}$ where $a_{\textrm{st}}$ is a uniform limit of $2$-step nilsequences and $a_{\textrm{er}}$ is a nullsequence?
\end{problem}
Under some extra ergodicity asssumptions, we shall establish the following theorem as a partial answer to Question \ref{questionintro}:
\begin{theorem}\label{intro1}
	Let $d \in \N$, let $(X,\mathcal{B},\mu)$ be a probability space, and $T_1,\dots,T_d: X \rightarrow X$ be commuting, invertible measure preserving transformations. Let $f_0,\dots,f_d \in L^{\infty}(\mu)$, and suppose that the transformations $T_1,\dots,T_d$ and $T_iT_j^{-1}$ are ergodic (for $1 \leq i\neq j \leq d$). Then, the multicorrelation sequence
	\begin{equation}\label{multi1} a(n):=\int_X f_0\cdot T_1^nf_1 \cdot \dotso \cdot T_d^n f_d \ d\mu
	\end{equation}
	can be decomposed as a sum of a uniform limit of $d$-step nilsequences ($a_{\textrm{st}}$) and a nullsequence ($a_{\textrm{er}}$).
\end{theorem}
It is sensible to ask whether Leibman's decomposition for polynomial multicorrelation sequences generalizes to the case of commuting transformations. A result in this direction was provided by Host and Frantzikinakis in Theorem 2.5 in \cite{hostfra}, where the authors showed that for any $\varepsilon>0$, any polynomial multicorrelation sequence of the form
\begin{equation} \tilde{a}(n):=\int_X f_0 \cdot T_1^{p_1(n)}f_1\cdot \dotso \cdot T_d^{p_d(n)}f_d \ d\mu \label{multi}\end{equation} 
admits a decomposition $\tilde{a}=\tilde{a}_{\textrm{st}}+\tilde{a}_{\textrm{er}}$, where $\tilde{a}_{\textrm{st}}$ is a uniform limit of $\ell$-step nilsequences for some $\ell \in \N$, and the error term satisfies $||\tilde{a}_{\textrm{er}}||_2 \leq \varepsilon$. While removal of $\varepsilon$ in $||a_{\textrm{er}}||_2$ in the appropriate generality remains a challenge, under certain ergodicity conditions, our methods can also be pushed to prove that the $\varepsilon$ can be removed. Indeed, we will show
\begin{theorem}\label{intro2bis}
	Let $(X,\mathcal{B},\mu)$ be a probability space and let $T_1,\dots,T_d: X \rightarrow X$ be measure preserving transformations. Assume that the measure preserving transformations $T_1^{a_1}\cdot\dotso\cdot T_d^{a_d}$ are ergodic for all $(a_1,\dots,a_d) \neq \vec{0}$. Let $p_1,\dots,p_k: \Z \rightarrow \Z^d$ be polynomial functions such that $p_i,p_i-p_j$ are not constant for all $i$ and for all $j \neq i$. Let $f_0, f_1,\dots,f_k \in L^{\infty}(\mu)$. Then the sequence 
	\[ a(n):=\int_X f_0\cdot T_{p_1(n)}f_1\cdot \dotso \cdot T_{p_k(n)}f_k \ d\mu\]
	is a sum of a uniform limit of $\ell=\ell(d,p_1,\dots,p_k)$-step nilsequences ($a_{\textrm{st}}$) and a nullsequence $(a_{\textrm{er}}$), where if $p_r(n)=(p_{r,1}(n),\dots,p_{r,d}(n))$, we write $T_{p_r(n)}f_r$ to mean $\prod_{i=1}^d T_i^{p_{r,i}(n)}f_r$.
\end{theorem}
We remark that the number $\ell$ depends on the number of applications of the van der Corput trick, essential in the PET induction introduced in \cite{wmpet} (see also \cite{bl1}, where it is used for commuting transformations). It is worth point out that $\ell$ admits an upper bound that only depends on $d, k$ and the highest degree of the polynomials in $\{p_1,\dots,p_k\}$
\\ \\
The ergodicity assumptions enable us to use the results on characteristic factors for commuting transformations that was established in the work of Frantzikinakis and Kra \cite{frakra}. Their work provides an essential tool in establishing the desired decomposition of the multicorrelation sequence $a(n)$. The main strategy of the proof builds, then, on previous work of Bergelson, Host and Kra in \cite{bhk} and Leibman in \cite{leibman1}, adapted for commuting transformations. 
\\ \\
Another purpose of this article is to prove two more results that showcase how natural extra ergodicity assumptions on a measure preserving system make ``expected'' results on ``large intersections'' hold. The first new result we establish on large intersections is the Theorem:
\begin{theorem}\label{intro2}
	Let $(X,\mathcal{B},\mu)$ be a measure space and let $T_1,T_2: X \rightarrow X$ be invertible measure preserving transformations such that $T_1, T_2, T_1T_2^{-1}$ are ergodic. Then, for all $\varepsilon>0$ and all $A \in \mathcal{B}$, the set
	\begin{equation}\label{length1}
		\{ n \in \Z : \mu(A \cap T_1^{-n}A \cap T_2^{-n}A)>\mu(A)^3-\varepsilon \}
	\end{equation}
	is syndetic (i.e., has bounded gaps in $\Z$).
\end{theorem}
Theorem \ref{intro2} should be compared and contrasted with the results of Chu in \cite{chu}: Chu showed that for an arbitrary ergodic measure preserving system $(X,\mathcal{B},\mu,T_1,T_2)$ one has that $\{ n \in \Z: \mu(A \cap T_1^{-n}A \cap T_2^{-n}A)>\mu(A)^4-\varepsilon\}$ is syndetic, but that one cannot improve the exponent $4$ to $3$ in the general setting. In fact, more is true: the exponent $4$ cannot be improved to any exponent strictly smaller than $4$ (see Theorem 1.2 in \cite{sundonoso}). Theorem \ref{intro2} gives some natural directional ergodicity conditions that allow to replace $\mu(A)^4$ with $\mu(A)^3$.
\\ \\
For arbitrarily large $d$, one certainly cannot hope for the set $\{ n \in \Z: \mu(A \cap T_1^{-n}A \cap \dots \cap T_d^{-n}A)>\mu(A)^d-\varepsilon\}$ to be syndetic even with the mild ergodicity conditions that make Theorem \ref{intro1} work. Indeed, examples in the appendix with I. Rusza in \cite{bhk} show that we can at most hope for such a result to hold if $d \leq 3$. Theorem 1.3 in \cite{sundonoso} shows that in fact, even $d=3$ is too much to hope for, as they provide an example of a measure preserving system $(X,\mathcal{B},\mu,T_1,T_2,T_3)$ with each $T_i$ ergodic and a set $A \in \mathcal{B}$ with $\mu(A)>0$ such that for every $\ell>0$,
\begin{equation}
\mu(A \cap T_1^{-n}A \cap T_2^{-n}A \cap T_3^{-n}A)<\mu(A)^{\ell}, \textrm{ for all } n \neq 0,
\end{equation}
so there is no more room for improvement in that direction. 
\\ \\
To prove Theorem \ref{intro2} we use the results on seminorms of Host \cite{host1} and the structural work in Frantzikinakis and Kra \cite{frakra} to establish a limiting formula (see Theorem \ref{twotranslimit} below), as well as a version of a trick of Frantzikinakis (\cite{fra2}, Proposition 5.1) which was employed in Section 8 of \cite{btz}.
\\ \\
Lastly, we also prove a variant of Theorem \ref{intro2} for shifted primes (i.e. $\mathbb{P}\pm 1$), using techniques from \cite{le}, \cite{fhk} and \cite{dmls} combined with our previous decomposition result, namely Theorem \ref{intro1}. In the statements and proofs, we will stick to $\mathbb{P}-1$ for convenience, but one easily checks that the same results hold for $\mathbb{P}+1$.
\begin{theorem}\label{intro3}
	Let $(X,\mathcal{B},\mu,T,S)$ be a measure preserving system with $T, S, TS^{-1}$ totally ergodic. Then, for all $\varepsilon>0$ and all $A \in \mathcal{B}$, the set 
	\[ \{n \in \Z : \mu(A \cap T^{-(p_n-1)}A \cap S^{-(p_n-1)} A)>\mu(A)^3-\varepsilon\} \]
	has positive lower density. 
\end{theorem}
The structure of the paper is as follows. In Section 2 we introduce the necessary notation and terminology that will be used throughout the paper, and recall a few basic results from Ergodic Theory. In Section 3 we obtain some auxiliary results that naturally extend the previously known facts pertaining to a single transformation.
\\ \\
In Section 4 we give the proof of Theorem \ref{intro1}, as well as its extension Theorem \ref{intro2bis}. Finally, in Section 5 we prove Theorem \ref{intro2} and Theorem \ref{intro3}.
\section{Preliminaries and Notation}
\subsection{Basic facts from Ergodic Theory}
\emph{Systems.} A measure preserving system is a tuple $(X,\mathcal{B},\mu,T_1,\dots,T_d)$, where $(X,\mathcal{B},\mu)$ is a standard Lebesgue probability space (see Definition 2.3 in \cite{walters}), and $T_1,\dots,T_d$ are commuting invertible measure preserving selfmaps of $X$.
\\ \\
\emph{Factors.} A \emph{homomorphism} from $(X,\mathcal{B},\mu,T_1,\dots,T_d)$ to $(Y,\mathcal{C},\nu,S_1,\dots,S_d)$ is a measurable map $\pi: X' \rightarrow Y'$ with $X'$ a $T_i$-invariant set of full measure and $Y'$ an $S$-invariant set of full measure (for all $i$), that intertwines $T_i$ and $S_i$ for all $i$. When such a map $\pi$ exists, we say that $(Y,\mathcal{C},\nu,S_1,\dots,S_d)$ is a factor of $(X,\mathcal{B},\mu,T_1,\dots,T_d)$. A factor can be identified with the $T_i$-invariant (for all $i$) sub-$\sigma$-algebra $\pi^{-1}(\mathcal{Y})$, and one can show that any $T_i$-invariant sub-$\sigma$-algebra of $\mathcal{B}$ defines a factor. Alternatively, a factor can also be thought of as a $T_i$-invariant subalgebra $\mathcal{F}$ of $L^{\infty}(X,\mathcal{B},\mu)$. In the sequel we will make extensive use of the \emph{Kronecker factor} of a system which is the smallest factor that makes all the $L^2(\mu)$-eigenfunctions measurable.
\\ \\
\emph{Conditional expectation.} If $\mathcal{Y}$ is a $T$-invariant sub-$\sigma$-algebra of $\mathcal{B}$ and $f \in L^1(\mu)$, we write $\mathbb{E}[f|\mathcal{Y}]$ for the conditional expectation of $f$ with respect to $\mathcal{Y}$. We will make use of the identities
\[ \int_X \mathbb{E}[f|\mathcal{Y}] \ d\mu=\int_X f \ d\mu \textrm{ and } T\mathbb{E}[f|\mathcal{Y}]=\mathbb{E}[Tf|\mathcal{Y}] \textrm{ (this equality holds for a.e. } x \in X).\]
\\ \\
\emph{Characteristic factors} Let $(F_N)$ be a F\o lner sequence in $\Z$ (i.e. a sequence of finite sets of $\Z$ with $\frac{|(F_N+k)\Delta F_N|}{|F_N|} \to 0$ as $N \to \infty$, for all $k \in \Z$) and let $p_1,\dots,p_k: \Z  \rightarrow \Z^d$ be a family of essentially distinct polynomials (i.e., such that $p_i-p_j$ is not constant for any $i \neq j$). We say that a sub-$\sigma$-algebra $\mathcal{D}$ of $\mathcal{B}$ is a characteristic factor of $(X,\mathcal{B},\mu,T_1,\dots,T_d)$ if for all $f_1,\dots,f_k \in L^{\infty}(\mu)$ we have
\[ \lim_{N \to \infty} \left\{\frac{1}{|F_N|}\sum_{n \in F_N}T_{p_1(n)}f_1\cdot \dotso \cdot T_{p_k(n)}f_k-\frac{1}{|F_N|}\sum_{n \in F_N}T_{p_1(n)}\mathbb{E}[f_1|\mathcal{D}]\cdot \dotso \cdot T_{p_k(n)}\mathbb{E}[f_k|\mathcal{D}]\right\}=0, \]
where the equality takes place in $L^2(\mu)$. We recall here that throughout the paper, if $p_r(n)=(p_{r,1}(n),\dots,p_{r,d}(n))$, we write $T_{p_r(n)}f_r$ to mean $\prod_{i=1}^d T_i^{p_{r,i}(n)}f_r$
\\ \\
\emph{Inverse limits.} We say that $(X,\mathcal{B},\mu,T_1,\dots,T_d)$ is an \emph{inverse limit sequence of factors} $(X,\mathcal{B}_j,\mu,T_1,\dots,T_d)$ if $(\mathcal{B}_j)_{j \in \N}$ is an increasing sequence of $T_i$-invariant (for all $i$)sub-$\sigma$-algebras such that $\bigvee_{j \in \N} \mathcal{B}_j=\mathcal{B}$ (here $\bigvee_{j \in \N} \mathcal{B}_j$ is the $\sigma$-algebra generated by the union of the $\mathcal{B}_j$) up to sets of measure zero.  
\\ \\

\emph{Ergodicity and the ergodic decomposition.} 
Given a measure preserving transformation $T: X\rightarrow X$, we denote by $\mathcal{I}(T)=\{A \in \mathcal{B} : \mu(A \Delta T^{-1}A)=0\}$ the algebra of invariant sets. We say that $T$ is ergodic if any set in $\mathcal{I}(T)$ has measure $0$ or $1$. The pointwise ergodic theorem states that if $T$ is ergodic, then for $\mu$-a.e. $x \in X$ we have
\[ \lim_{N\to \infty} \frac{1}{N} \sum_{n=1}^N f(T^nx)=\int_X f \ d\mu. \]
There exists a map $x \mapsto \mu_x$ that is $\mathcal{I}(T)$-measurable and such that for every $f \in L^{\infty}(\mu)$ we have
\[ \mathbb{E}[f|\mathcal{I}(T)](x)=\int_X f \ d\mu_x, \textrm{ for } \mu \textrm{ almost every } x \in X.\]
The \emph{ergodic decomposition} for the measure $\mu$ is given as an integral of $\mu_x$ with respect to a probability measure $\lambda$ sitting on the space of extreme points of $T$-invariant measures, denoted by $\Omega$
\[ \mu:=\int_{\Omega} \mu_{\omega} \ d\lambda(\omega).\]
Moreover, the measures $\mu_x$ are such that for $\mu$-a.e. $x \in X$, the system $(X,\mathcal{B},\mu_x,T)$ is ergodic.
\subsection{The Gowers-Host-Kra seminorms $||| \cdot |||_k$ and the factors $\mathcal{Z}_k$.}
\emph{The Gowers-Host-Kra seminorms $|||\cdot|||_k$.} Let $(X,\mathcal{B},\mu,T)$ be a measure preserving system. We use $\mu \times_{\mathcal{I}(T)} \mu$ to denote the relatively independent joining of $(X,\mathcal{B},\mu)$ with itself with respect to the $\sigma$-algebra $\mathcal{I}(T)$ 
(i.e. $\nu \times_{\mathcal{I}(T)} \nu$ is a measure on $X \times X$ satisfying $\int_{X \times X} f_1 \otimes f_2  \ d\mu \times_{\mathcal{I}(T)} \mu=\int_X \mathbb{E}[f_1|\mathcal{I}(T)]\mathbb{E}[f_2|\mathcal{I}(T)] \ d\mu$ for all $f_1,f_2 \in L^{\infty}(\nu)$)
 Write $\mu=\int \mu_x \ d\lambda$ for the ergodic decomposition of $\mu$. For every $k \geq 1$ we define a measure $\mu^{[k]}$ on $X^{2^k}$ invariant under $T^{[k]}:=T \times \dots \times T$ ($2^k$ times) by 
\[ \mu^{[1]}=\int \mu_x \times \mu_x \ d\mu, \textrm{ and inductively}\]
\[ \mu^{[k+1]}=\mu^{[k]} \times_{\mathcal{I}(T^{[k]})} \mu^{[k]} \textrm{ for } k \geq 1).	\]
We can now define a seminorm (see Lemma 3.9 of \cite{host1}) $|||\cdot|||_k$ on $L^{\infty}(\mu)$ by
\[|||f|||_k:= \left(\int_{X^{2^k}} \prod_{i=0}^{2^k-1}\mathcal{C}^{\sum_{j}a_j(i)}(f(x_i))\ d\mu^{[k]}\right)^{1/2^k},\]
where $i=\sum_{j}a_j(i)2^j$ is the binary expansion of $i \in [0,2^{k}-1]$. Throughout we use the convention 
\begin{equation}
\mathcal{C}^kz:=\begin{cases}
z, \textrm{ if } k \textrm{ even, } \\
\bar{z}, \textrm{ otherwise.}
\end{cases}
\end{equation}
\emph{The factors $\mathcal{Z}_k$.} For every $k \geq 1$, the factors $\mathcal{Z}_k$ denote the invariant $\sigma$-algebras on $X$ satisfying the property
\[\textrm{ for } f \in L^{\infty}(\mu), \  \mathbb{E}[f|\mathcal{Z}_{k-1}]=0 \textrm{ if and only if }|||f|||_k=0.\]
See \cite{hostkrabook} for more details.
\subsection{Structure theory and nilsystems} A \emph{nilmanifold} is a homogeneous space $X=G/\Gamma$ where $G$ is a nilpotent Lie group, and $\Gamma$ is a discrete cocompact subgroup of $G$. Let $G_k$ be the $k$-th commutator subgroup of $G$. If $G_{k+1}=\{e\}$, we say that $G$ is a $k$-step nilpotent Lie group. We say that $X=G/\Gamma$ is a $k$-\emph{step nilmanifold} if $G$ is a $k$-step nilpotent Lie group.

A $k$-step nilpotent Lie group $G$ acts on $G/\Gamma$ by left translations, where the translation by a fixed element $a \in G$ is given by $T_a(g\Gamma)=(ag)\Gamma$. By $m_X$ we denote the unique probability measure on $X$ that is invariant under the action of $G$ by left translations, and by $\mathcal{G}/\Gamma$ we denote the Borel $\sigma$-algebra of $G/\Gamma$. Fixing an element $a \in G$, we call the measure preserving system $(G/\Gamma,\mathcal{G}/\Gamma,m_X,T_a)$ a $k$-\emph{step nilsystem}.

If $X=G/\Gamma$ is a $k$-step nilmanifold, $a \in G$, $x \in X$, and $f \in C(X)$, we call the sequence $(f(a^nx))_{n \in \N}$ a $k$-\emph{step nilsequence}. We caution the reader that it is possible to find slightly different usage of the terminology we just presented in the literature for nilsequences. Regarding the previously introduced nullsequences, we will use the following equivalent formulations: 
First, since $a$ is bounded, it follows that the square in \eqref{besicovitch} is immaterial, so $a$ is a nullsequence if and only if
\[ \lim_{N-M \to \infty} \frac{1}{N-M} \sum_{n=M}^{N-1} |a(n)|=0.\]
Second, an equivalent condition to \eqref{besicovitch} is 
\[ \lim_{N \to \infty} \frac{1}{|F_N|}\sum_{n \in F_N} |a(n)|=0\]
for all F\o lner sequences $(F_N)$ in $\Z$.

We will need later the following fundamental result of Host and Kra:
\begin{theorem}[\cite{hostkraoriginal}] Let $(X,\mathcal{B},\mu,T)$ be an ergodic measure preserving system and $k \in \N$. Then the measure preserving system $(X,\mathcal{Z}_k,\mu,T)$ is a (measure theoretic) inverse limit of $k$-step nilsystems.
\end{theorem}
\subsection{Gowers seminorms and $o$-notation}
We will use below the notation $o_{N\to \infty, a_1,\dots,a_k}(1)$, which denotes a quantity depending on $N, a_1,\dots, a_k$ that goes to $0$ as $N \to \infty$ if we fix the quantities $a_1,\dots,a_k$. This convergence is assumed to be uniform on all other quantities that do not explicitly appear in the subindex.
\\ \\
In order to define the Gowers seminorms, we find it convenient to use expectation on finite sets as follows: if $A$ is a finite set, the expectation of a function $f: A \rightarrow \C$ is defined as the following average:
\[ \mathbb{E}_{n \in A}f(n)=\frac{1}{|A|}\sum_{n \in A} f(n).\]
The Gowers seminorms of $f: \Z/N\Z \rightarrow \C$ are inductively defined as follows: 
\[ ||f||_{U_1(\Z/N\Z)}:=|\mathbb{E}_{n \in \Z/N\Z} f(n)|\]
and
\[ ||f||_{U_{d+1}(\Z/N\Z)}:=\left(\mathbb{E}_{h \in \Z/N\Z} ||f_h \cdot \bar{f}||_{U_d(\Z/N\Z)}^{2^d}\right)^{1/2^{d+1}},\]
where $f_h(n):=f(n+h)$ for all $h \in \Z/N\Z$. Gowers showed in \cite{gowers} that $||\cdot||_{U_d}$ is a seminorm for all $d \in \N$.
\section{Technical results}
In this section we introduce a natural variant of Gowers-Host-Kra seminorms for the case of commuting transformations and establish some basic facts which will be needed for the proofs in subsequent sections.
\\ \\
We begin by reviewing a few key concepts that appeared first in Host's article \cite{host1} in an abridged form. 
\subsection{Notation}
We will assume all functions in the following subsections to be measurable and complex-valued. For $k \in \N$, the points in $X^{2^k}$ are written 
\[x=(x_{\varepsilon} : \varepsilon \in \{0,1\}^k).\]
Our convention is to write elements in $\{0,1\}^k$ without commas and parentheses, so that we can append a $0$ or $1$ to it and have them belong to $\{0,1\}^{k+1}$.
If $f_{\varepsilon}, \varepsilon \in \{0,1\}^k$ are functions on $X$, we define a function on $X^{2^k}$ by
\[\left(\bigotimes_{\varepsilon \in \{0,1\}^k} f_{\varepsilon}\right)(x):=\prod_{\varepsilon \in \{0,1\}^k}f_{\varepsilon}(x_{\varepsilon}). \]
Let $d \in \N$ and put $X^*=X^{2^d}$. For $1 \leq i \leq d$, let $T_i^{\Delta}=T_i\times \dots \times T_i$, the diagonal transformation of $X^{2^k}$, so that $(T_i^{\Delta}x)_{\varepsilon}=T_ix_{\varepsilon}$ for all $\varepsilon \in \{0,1\}^k$.
\\ \\
The \emph{side} transformations $T_i^*$ of $X^*$ are given by 
\[ \textrm{ for every } \varepsilon \in \{0,1\}^d, \quad (T_i^*x)_{\varepsilon}=
\begin{cases}
T_ix_{\varepsilon} & \textrm{ if } \varepsilon_i=0; \\
x_{\varepsilon} & \textrm{ if } \varepsilon_i=1.
\end{cases}
\]
\subsection{The box measure and its associated seminorm}
The box measure $\mu^*$ on $X^*$ is constructed inductively as follows. First, define a measure on $X \times X$ by
\[\mu_{T_1}:=\mu \times_{\mathcal{I}(T_1)} \mu,\]
and then, for all $k \leq d-1$, define the measures $\mu_{T_1,\dots,T_{k+1}}$ on $X^{2^{k+1}}$ inductively by
\[\mu_{T_1,\dots,T_{k+1}}=\mu_{T_1,\dots,T_k} \times_{\mathcal{I}(T_{k+1}^{[k]})} \mu_{T_1,\dots, T_k}.\]
Finally, set $\mu^*=\mu_{T_1,\dots,T_d}$ on $X^*$. One can check that $\mu^*$ is invariant under the diagonal and side transformations of $X^*$ as defined above.
\\ \\
The ergodic theorem and downwards induction imply the following result:
\begin{lemma}[see \cite{host1}]
	Let $f_{\varepsilon}, \varepsilon \in \{0,1\}^d$ be functions in $L^{\infty}(\mu)$. Then,
	\[ \int \prod_{\varepsilon \in \{0,1\}^d}\mathcal{C}^{|\varepsilon|}(f_{\varepsilon}(x_{\varepsilon})) \ d\mu^*=\lim_{N_d \to \infty}\frac{1}{|I_{N_d}|}\sum_{n_d \in I_{N_d}}\dots\lim_{N_1 \to \infty}\frac{1}{|I_{N_1}|}\sum_{n_1 \in I_{N_1}} \int \prod_{\varepsilon \in \{0,1\}^d} \mathcal{C}^{|\varepsilon|} (T_1^{(1-\varepsilon_1)n_1}\dots T_d^{(1-\varepsilon_d)n_d}f_{\varepsilon})\ d\mu. \]
	for any sequences of intervals $(I_{N_j})_{N_j \in \N}$, $1 \leq j \leq d$ with lengths going to infinity.
\end{lemma}
Host then shows that for every $f \in L^{\infty}(\mu)$ one has
\begin{equation}\label{positivity}\int \prod_{\varepsilon \in \{0,1\}^d}\mathcal{C}^{|\varepsilon|}(f(x_{\varepsilon})) \ d\mu^* \geq 0.
\end{equation}
Alternatively, one can show the inequality \eqref{positivity} using an argument similar to that in the proof of Theorem 0.5 in \cite{bl}. The inequality \eqref{positivity} then allows the definition
\begin{definition} For $f \in L^{\infty}(\mu)$ let
	\[|||f|||_{T_1,\dots,T_d}:=\left(\int \prod_{\varepsilon \in \{0,1\}^d}\mathcal{C}^{|\varepsilon|}(f(x_{\varepsilon})) \ d\mu^* \right)^{1/2^d}.\]
\end{definition}
It is easily checked that
\[|||f|||_{T_1,\dots,T_d}^{2^d}=\lim_{N_d \to \infty}\frac{1}{|I_{N_d}|}\sum_{n \in I_{N_d}}|||T_d^{n}f \cdot \bar{f}|||_{T_1,\dots,T_{d-1}}^{2^{d-1}},\]
where $(I_{N_d})_{N \in \N}$ is any F\o lner sequence in $\Z$.
\begin{proposition}[Host \cite{host1}]
	\begin{itemize}
		\item[(i)] For $f_{\varepsilon} \in L^{\infty}(\mu)$ for all $\varepsilon \in \{0,1\}^d$ we have
		\[\left| \int \bigotimes_{\varepsilon \in \{0,1\}^d} \mathcal{C}^{|\varepsilon|}f_{\varepsilon} \ d\mu^*\right| \leq \prod_{\varepsilon \in \{0,1\}^d} |||f_{\varepsilon}|||_{T_1,\dots,T_d},\]
		\item[(ii)] $|||\cdot|||_{T_1,\dots,T_d}$ is a seminorm on $L^{\infty}(\mu)$.
	\end{itemize}
\end{proposition}
\subsection{Technical results on seminorms for commuting transformations}
In this subsection we formulate technical results which will be needed in the sequel:
\begin{proposition}[Host \cite{host1}]\label{seminorm1} Let $(X,\mathcal{B},\mu,S_1,\dots,S_d)$ be a measure preserving system. Let $T_1=S_1$ and $T_i=S_1S_i^{-1}$ for $2 \leq i \leq d$. Then, for every $f_1,\dots,f_d \in L^{\infty}(\mu)$ with $||f_i||_{L^{\infty}(\mu)} \leq 1$ for $2 \leq i \leq d$ we have
	\[ \limsup_{N \to \infty} \left|\left| \frac{1}{|F_N|}\sum_{n \in F_N} S_1^n f_1 \cdot \dotso \cdot S_d^nf_d \right|\right|_{L^2(\mu)} \leq |||f_1|||_{T_1,\dots,T_d},\]
	for every F\o lner sequence $(F_N)_{N \in \N}$ in $\Z$.
\end{proposition}
\begin{theorem}[Corollary 3 \cite{host1}]\label{permuting} The seminorm $|||\cdot|||_{T_1,\dots,T_d}$ remains unchanged if the transformations $T_1,\dots,T_d$ are permuted.
\end{theorem}
The following lemma is given for complex valued functions, as the argument does not require significant changes:
\begin{lemma}\label{lem1}
	Let $d \in \N$. Let $(X,\mathcal{B},\mu,T_1,\dots,T_d)$ be a measure preserving system. Let $f \in L^{\infty}(\mu)$. Then,
	\[ |||f \otimes \bar{f}|||_{T_1 \times T_1,\dots,T_d \times T_d} \leq ||| f|||_{T_1,\dots,T_d,T_i}^2,\]
	for all $1 \leq i \leq d$.
\end{lemma}
\begin{proof}
	We proceed by induction on $d \in \N$. Let $d=1$. We are going to show that\\ $|||f \otimes \bar{f}|||_{T_1 \times T_1} \leq |||f|||_{T_1,T_1}^2$. Consider the ergodic decomposition of $\mu$ with respect to $T_1$:
	\begin{equation}\label{ergdec1} \mu=\int_{\Omega} \mu_{\omega} \ d\lambda(\omega). \end{equation}
	We have
	\[ ||| f \otimes \bar{f}|||_{T_1 \times T_1}^2=\lim_{N \to \infty} \frac{1}{N} \sum_{n=1}^N \int_{X^2} (T_1^nf \otimes T_1^n \bar{f})\bar{f} \otimes f \ d\mu \ d\mu=\lim_{N \to \infty}\frac{1}{N} \sum_{n=1}^N\left| \int_X \bar{f} T_1^n f \ d\mu \right|^2= \]
	(using \eqref{ergdec1})
	\[ \lim_{N \to \infty} \frac{1}{N} \sum_{n=1}^N \left|\int_{\Omega}\int_X \bar{f} T_1^n f \ d\mu_{\omega} \ d\lambda \right|^2 \leq \int_{\Omega} \lim_{N \to \infty} \frac{1}{N} \sum_{n=1}^N \left| \int_X \bar{f} T_1^nf \ d\mu_{\omega} \right|^2 \ d\lambda, \]
	where we used the Cauchy-Schwarz inequality first, then linearity of $\int_{\Omega}$ and the Dominated Convergence Theorem. Next, rewriting the expression in the right hand side of the formula above we obtain
	\[ \int_{\Omega} \lim_{N \to \infty} \frac{1}{N} \sum_{n=1}^N \int_{X^2} (T_1^nf \otimes T_1^n \bar{f} )\bar{f} \otimes f \ d\mu_{\omega} \ d\mu_{\omega} \ d\lambda\]
	which equals 
	\begin{equation} \int_{\Omega} \int_{X^2} \mathbb{E}_{\mu_{\omega} \otimes \mu_{\omega}}[f \otimes \bar{f} \mid \mathcal{I}(T_1 \times T_1) \mathbb{E}_{\mu_{\omega} \otimes \mu_{\omega}}[\bar{f} \otimes f \mid \mathcal{I}(T_1 \times T_1)] \ d\mu_{\omega} \ d\mu_{\omega} \ d\lambda \label{erg1}\end{equation}
	(using the ergodic theorem, and properties of conditional expectation). Finally, we notice that since for a.e. $\omega \in \Omega$, $\mu_{\omega}$ is ergodic, then, the definition of the seminorms simplifies a little bit, so by Proposition 18 in Chapter 8 of \cite{hostkrabook} we can rewrite equation \eqref{erg1} as
	\[ \int_{\Omega} |||f|||_{\mu_{\omega},T_1,T_1}^4 \ d\lambda=|||f|||_{T_1,T_1}^4. \]
	So we showed $|||f \otimes \bar{f}|||^2_{T_1 \times T_1} \leq |||f|||_{T_1,T_1}^4$. Taking square roots completes the proof of the base case.
	\\ \\
	So now suppose that the result in question holds for $d=d_0 \geq 1$, and consider
	\[ ||| f \otimes \bar{f}|||_{T_1 \times T_1, \dots, T_{d_0+1} \times T_{d_0+1}}^{2^{d_0+1}}=\lim_{N \to \infty}\frac{1}{N}\sum_{n=1}^N |||T_{d_0+1}^n f \cdot \bar{f} \otimes T_{d_0+1}^n \bar{f} \cdot f|||^{2^{d_0}}_{T_1 \times T_1, \dots, T_{d_0} \times T_{d_0}} \leq\]
	(using the inductive hypothesis)
	\[ \limsup_{N \to \infty} \frac{1}{N}\sum_{n=1}^N |||T_{d_0+1}^nf \cdot \bar{f}|||_{T_1,\dots,T_{d_0},T_i}^{2^{d_0+1}}=|||f|||_{T_1,\dots,T_{d_0},T_{i},T_{d_0+1}}^{2^{d_0+2}}\]
	Finally, note that seminorms for commuting transformations do not depend on the order of the transformations involved (see Corollary 3 in \cite{host1}), so that permuting them in the expression above, gives the desired result.
\end{proof}
\begin{lemma}\label{lem2}
	Let $d \in \N$. Let $(X,\mathcal{B},\mu,T_1,\dots,T_d)$ be a measure preserving system. Let $f \in L^{\infty}(\mu)$. Suppose that the measure preserving transformations $T_1,\dots,T_d$ are ergodic. Then, we have 
	\begin{equation}\label{seminormequality1} |||f|||_{T_1,\dots,T_d}=|||f|||_{\underbrace{T_i,T_i,\dots,T_i}_d},\end{equation}
	for all $1 \leq i \leq d$.
\end{lemma}
\begin{proof}
	First, notice that by Theorem \ref{permuting}, it suffices to show this result in the case where $i=1$, and then use symmetry, so let $i=1$. Next, to show \eqref{seminormequality1} it suffices to prove that $\mu_{T_1,\dots,T_d}=\mu_{T_1,\dots,T_1}$. Recall that the measures $\mu_{T_1,\dots,T_d}$ are defined inductively so that
	\[ \mu_{T_1,\dots,T_d}=\mu_{T_1,\dots,T_{d-1}} \times_{\mathcal{I}(T_d^{[d-1]})} \mu_{T_1,\dots,T_{d-1}},\]
	(see the subsection on the construction of $\mu^*$) so we can change $\mathcal{I}(T_d \times \dots \times T_d)$ for $\mathcal{I}(T_1 \times \dots \times T_1)$ (see Proposition 29 of Chapter 8 in \cite{hostkrabook}), so we get
	\[ \mu_{T_1,\dots,T_d}=\mu_{T_1,\dots,T_{d-1},T_1}.\]
	From this, it follows that $|||f|||_{T_1,\dots,T_d}=|||f|||_{T_1,T_1,\dots,T_{d-1}}$.
	\\ \\
	Repeating this argument using our ergodicity assumptions, we are allowed to change all the other transformations for $T_1$, which implies the result.
\end{proof}
Next, we give an essential theorem that characterizes the $\mathcal{Z}_k$ factors.
\begin{definition} We say that a measure preserving system $(X,\mathcal{B},\mu,T_1,\dots,T_d)$ is \emph{toral of order} $k$ if it is isomorphic to $(G/\Gamma,\textrm{Borel}(G/\Gamma),\mu_{\textrm{Haar}},T_{a_1},\dots,T_{a_d})$, where $G$ is a $k$-step nilpotent Lie group, $\Gamma$ a cocompact subgroup, and the transformations $T_{a_1},\dots,T_{a_d}$ act by ergodic niltranslations by commuting elements $a_1,\dots a_d \in G$ on $G$.
\end{definition}
\begin{definition}
	We say that a measure preserving system $(X,\mathcal{B},\mu,T_1,\dots,T_d)$ is \emph{of order} $k$ if $X=Z_k(X)$ and each $T_i$ is ergodic. (Note that the notion $Z_k(X)$ is independent of the transformations $T_i$ because of the ergodicity assumptions made on $T_i$.)
	\end{definition}
\begin{theorem}[Frantzikinakis-Kra \cite{frakra}]\label{frak}
	Let $(X,\mathcal{B},\mu,T_1,\dots,T_d)$ be a measure preserving system of order $k$. Then, the system is an inverse limit of a sequence $\{(X_i,\mathcal{B}_i,\mu_i,T_1,\dots,T_d)\}_{i \in \N}$ of toral systems of order $k$. Moreover, these toral systems of order $k$ are isomorphic to a $k$-step nilsystem $(G/\Gamma,m_{G/\Gamma},T_{a_1},\dots,T_{a_d})$.
\end{theorem}
We note that the sequence of factors in Theorem \ref{frak} will be denoted by $Z_i(X)$ as the number of commuting transformations involved do not change them.
We end this section with a result of Johnson that will allow us to push Theorem \ref{intro1} for multicorrelation sequences arising from general families of polynomials:
\begin{theorem}[Johnson \cite{mcrjohnson}]\label{mcrj}
	Let $(X,\mathcal{B},\mu,T_1,\dots,T_d)$ be a measure preserving system. Assume that $T_1^{c_1}\cdot\dotso\cdot T_d^{c_d}$ is ergodic for all $(c_1,\dots,c_d) \neq \vec{0}$. Let $p_1,\dots,p_d: \Z \rightarrow \Z^d$ be polynomial functions such that $p_i,p_i-p_j$ are not constant for all $i$ and for all $j \neq i$. Let $f_1,\dots,f_d \in L^{\infty}(\mu)$. Then,
	\[ \lim_{N-M \to \infty} \frac{1}{N-M}\sum_{n=M}^{N-1}T_{p_1(n)}f_1\cdot \dotso \cdot T_{p_d(n)}f_d=0\]
	in $L^2(\mu)$, provided $|||f_i|||_{\underbrace{T_i,\dots,T_i}_{\ell}}=0$ for any $1 \leq i \leq d$, for some suitable $\ell \in \N$ depending only on the given family of polynomials.
\end{theorem}
\section{Removal of $\varepsilon$ under some ergodicity assumptions}
We now move to the proof of Theorem \ref{intro1} and its generalization to polynomial multicorrelation sequences, which we state here again for the convenience of the reader.
\begin{theorem}\label{mainthm}
	Let $d \in \N$ and $(X,\mathcal{B},\mu,T_1,\dots,T_d)$ a measure preserving system. Let $f_0,\dots,f_d \in L^{\infty}(\mu)$, and suppose that the transformations $T_1,\dots,T_d$ and $T_iT_j^{-1}$ are ergodic (for $1 \leq i\neq j \leq d$). Then, the multicorrelation sequence
	\[ a(n):=\int_X f_0\cdot T_1^nf_1 \cdot \dotso \cdot T_d^n f_d \ d\mu\]
	can be decomposed as a sum of a uniform limit of $d$-step nilsequences ($a_{\textrm{st}}$) and a nullsequence ($a_{\textrm{er}}$).
\end{theorem}
\begin{proof} We follow and adapt the proof strategy in Section $3$ of Leibman's paper \cite{leibman1}. Without loss of generality, we assume that $||f_i||_{\infty} \leq 1$ for all $i$. Let $\Phi_N$ be a F{\o}lner sequence in $\Z$. Then, by Proposition \ref{seminorm1}, the Cauchy-Schwarz inequality, and the fact that $||f_0||_{\infty} \leq 1$, we have 
	\[ \lim_{N \to \infty} \frac{1}{|\Phi_N|} \sum_{n \in \Phi_N} \left| \int_X f_0 \cdot T_1^nf_1 \cdot \dotso \cdot T_d^nf_d \ d\mu \right|^2=\]
	\begin{multline}\label{bound1}\lim_{N \to \infty} \frac{1}{|\Phi_N|} \sum_{n \in \Phi_N} \int_{X^2} f_0 \otimes \bar{f_0} \cdot (T_1 \times T_1)^n (f_1 \otimes \bar{f_1})\cdot \dotso \cdot (T_d \times T_d)^n (f_d \otimes \bar{f_d}) \ d\mu \ d\mu \leq \\
	\lim_{N \to \infty} \left|\left|\frac{1}{|\Phi_N|}\sum_{n \in \Phi_N}(T_1 \times T_1)^n (f_1 \otimes \bar{f_1})\cdot \dotso \cdot (T_d \times T_d)^n (f_d \otimes \bar{f_d})\right|\right|_{L^2(\mu)} \leq
	|||f_i \otimes \bar{f_i}|||_{T_i \times T_i,((T_i \times T_i)^{-1}(T_j \times T_j))_{i \neq j}}, \end{multline}
	for all $1 \leq i \leq d$. Thus, given our ergodicity assumptions, and using Lemma \ref{lem1} and \ref{lem2}, we can bound from above the seminorm appearing in \eqref{bound1} by 
	\[ |||f_i|||_{T_i,T_i,(T_i^{-1}T_j)_{i\neq j}}^2=|||f_i|||_{T_i,\dots,T_i}^2=|||f_i|||_{d+1,T_i}^2.\]
	Therefore, for $1 \leq i\leq d$ we have
	\begin{equation}\label{bound2}
		\lim_{N \to \infty} \frac{1}{|\Phi_N|} \sum_{n \in \Phi_N} \int_{X^2} f_0 \otimes \bar{f_0} \cdot (T_1 \times T_1)^n (f_1 \otimes \bar{f_1})\cdot \dotso \cdot (T_d \times T_d)^n (f_d \otimes \bar{f_d}) \ d\mu \ d\mu \leq |||f_i|||_{d+1,T_i}^2
	\end{equation}
	The bound \eqref{bound2} and Theorem \ref{frak} imply that the sequence
	\begin{equation}\label{multinull}
		a(n)-\int_{Z_{d}(X)} f_0 \cdot T_1^n \E[f_1 \mid Z_{d}(X)] \cdot \dotso \cdot T_d^n\E[f_k \mid Z_{d}(X)] \ d\mu_{Z_{d}(X)} 
	\end{equation}
	is a null-sequence.
	\\ \\
	Let $\varepsilon>0$. By Theorem \ref{frak}, $Z_d(X)=(X_d,\textrm{Borel}(X_d),\mu_{X_d},T_1,\dots,T_d)$ is an inverse limit of nilsystems. Thus, there exists a factor of $Z_d(X)$ with the structure of a $d$-step nilsystem $(\tilde{X},\textrm{Borel}(\tilde{X}),\mu_{\tilde{X}},T_1,\dots,T_d)$, on which each $T_i$ acts by the niltranslation by an element $a_i \in \tilde{X}$, such that for $\tilde{f_i}=\E[f_i \mid \tilde{X}]$ we have
	\[ \left|\int_{X_d} f_0 \cdot T_1^n \E[f_1 \mid Z_{d}(X)] \cdot \dotso \cdot T_d^n\E[f_d \mid Z_{d}(X)] \ d\mu_{X_d}-\int_{\tilde{X}} \tilde{f_0} \cdot a_1^n\tilde{f_1} \cdot \dotso \cdot a_d^n \tilde{f_k} \ d\mu_{\tilde{X}} \right|<\varepsilon\]
	for all $n \in \Z$. Therefore, there exists a nullsequence $\lambda$ such that 
	\begin{equation}\label{ellinfty1}
		\left|\left|a(n)-\left(\int_{\tilde{X}} \tilde{f_0} \cdot a_1^n\tilde{f_1} \cdot \dotso \cdot a_d^n \tilde{f_k} \ d\mu_{\tilde{X}}+\lambda(n)\right)\right|\right|_{\ell^{\infty}(\Z)}<\varepsilon.
	\end{equation}
	A standard approximation argument allows us to assume without loss of generality that $\tilde{f_1},\dots,\tilde{f_d} \in C(\tilde{X})$ in \eqref{ellinfty1}. Applying Theorem 2.5 in \cite{leibman1} to the nilmanifold $\tilde{X}^k$, the diagonal subnilmanifold $\{(x,\dots,x) : x \in \tilde{X}\}$, the linear polynomial sequence $(a_1^n,\dots,a_d^n)$ and the function $f(x_1,\dots,x_d)=\tilde{f_1}(x_1)\cdot\dotso \cdot \tilde{f_d}(x_d) \in C(\tilde{X}^k)$, we obtain that the sequence
	\[\int_{\tilde{X}} \tilde{f_0} \cdot a_1^n\tilde{f_1} \cdot \dotso \cdot a_d^n \tilde{f_k} \ d\mu_{\tilde{X}} \]
	is a sum of a $d$-step nilsequence and a nullsequence.
	\\ \\
	Therefore, for each $\varepsilon>0$ we can find a $d$-step nilsequence $\psi$, a nullsequence $\lambda$ and a bounded sequence $\delta$ with $||\delta||_{\ell^{\infty}(\Z)} \leq \varepsilon$ such that
	\begin{equation}\label{decomp1} a(n)=\psi(n)+\lambda(n)+\delta(n).\end{equation}
	For each $l \in \N$, consider the decomposition $a=\psi_l+\lambda_l+\delta_l$, where $||\delta_l||_{\ell^{\infty}}<\frac{1}{l}$. For $r \neq l$, we have
	\begin{equation}\label{nilsequence1} |\psi_l(n)-\psi_r(n)|=|(\lambda_l(n)-\lambda_r(n))+(\delta_l(n)-\delta_r(n))|.
	\end{equation}
	Now, $\lim_{N-M \to \infty} \frac{1}{N-M}\sum_{n=M}^N |\lambda_l(n)-\lambda_r(n)|=0$ and $\sup_{n \in \Z} |\delta_r(n)-\delta_l(n)| \leq \frac{1}{l}+\frac{1}{r}$. Therefore,
	\begin{equation}\label{nilsequence2}
		|\psi_l(n)-\psi_r(n)|\leq \frac{1}{l}+\frac{1}{r}
	\end{equation}
	for all $n \in \Z$ except potentially a subset $A \subseteq \Z$ with $\mathbb{1}_A(n)$ a nullsequence. For each $l, r \in \N$, the sequence $\psi_l(n)-\psi_r(n)$ is a nilsequence, so it follows that inequality \eqref{nilsequence2} must, in fact, hold for all $n \in \Z$. Hence, the sequence $(\psi_l)_{l \in \N}$ is a Cauchy sequence in $\ell^{\infty}(\Z)$ that consists of $d$-step nilsequences, and since we already showed that $(\delta_r)_{r \in \N}$ is a Cauchy sequence in $\ell^{\infty}(\Z)$ converging to a nullsequence, we are done.
\end{proof}
The following theorem extends Theorem \ref{mainthm} to more general polynomial multicorrelation sequences, at the cost of more stringent ergodicity assumptions on the measure preserving system:
\begin{theorem}\label{mainthmgen}
	Let $(X,\mathcal{B},\mu,T_1,\dots,T_d)$ be a measure preserving system. Assume that the measure preserving transformations $T_1^{a_1}\cdot\dotso\cdot T_d^{a_d}$ are ergodic for all $(a_1,\dots,a_d) \neq \vec{0}$. Let $p_1,\dots,p_k: \Z \rightarrow \Z^d$ be polynomial functions such that $p_i,p_i-p_j$ are not constant for all $i$ and for all $j \neq i$. Let $f_0, f_1,\dots,f_k \in L^{\infty}(\mu)$. Then the sequence 
	\[ a(n):=\int_X f_0\cdot T_{p_1(n)}f_1\cdot \dotso \cdot T_{p_k(n)}f_k \ d\mu\]
	is a sum of a uniform limit of $\ell=\ell(d,p_1,\dots,p_k)$-step nilsequences ($a_{\textrm{st}}$) and a nullsequence $(a_{\textrm{er}}$).
\end{theorem}
\begin{proof}
	We can argue similarly to the proof of Theorem \ref{mainthm}. First, we reduce to the case (without loss of generality) where $||f_i||_{\infty} \leq 1$ for all $i$. Then, observe that for any F\o lner sequence $(F_N)_{N \in \N}$ in $\Z$, we have 
	\begin{multline}\label{av} \lim_{N \to \infty}\frac{1}{|F_N|}\sum_{n \in F_N} |a(n)|^2=\lim_{N \to \infty}\frac{1}{|F_N|}\sum_{n \in F_N}\int_{X^2} f_0 \otimes \bar{f_0}\cdot (T_{p_1(n)} \times T_{p_1(n)})f_1 \otimes \bar{f_1}\cdot \dotso \cdot (T_{p_k(n)} \times T_{p_k(n)})f_k \otimes \bar{f_k}\ d\mu \ d\mu \\ \leq 
	\lim_{N \to \infty} \left|\left| \frac{1}{|\Phi_N|}\sum_{n \in \Phi_N}(T_{p_1(n)} \times T_{p_1(n)})f_1 \otimes \bar{f_1}\cdot \dotso \cdot (T_{p_k(n)} \times T_{p_k(n)})f_k \otimes \bar{f_k} \right|\right|_{L^2(\mu)}
	\end{multline}
	By Theorem \ref{mcrj}, the last limit in \eqref{av} is bounded by a seminorm of the form $|||f_i \otimes \bar{f_i}|||_{\underbrace{T_i \times T_i,\dots,T_i \times T_i}_{r(i)}} \leq |||f_i|||^2_{\underbrace{T_i,\dots,T_i}_{r(i)+1}}$, for some $r_i \in \N$, for all $i \in \N$, by Theorem \ref{mcrj}. Therefore, there is a common $\ell=\ell(d,p_1,\dots,p_k)$, big enough, such that if any of the norms $|||f_i|||_{\ell}=0$, then the averages \eqref{av} converge to $0$. (These seminorms are well defined because of our extra ergodicity assumptions).
	\\ \\
	This implies that we can proceed as in the proof of Theorem \ref{mainthm} but now using the factor $Z_{\ell}(X)$ instead of $Z_d(X)$. Let $\varepsilon>0$. There exists an $\ell$-step nilsystem $\tilde{X}$, a factor of $X$, on which the transformations $T_i$ act by niltranslations by commuting elements $a_i \in \tilde{X}$, there exist $\tilde{f_0},\tilde{f_1},\dots,\tilde{f_k} \in C(\tilde{X})$ and a nullsequence $\lambda \in \ell^{\infty}(\Z)$ such that
	\[ \left| a(n)-\left(\lambda(n)+\int_{\tilde{X}} \tilde{f_0}\cdot \vec{a}_{p_1(n)} \tilde{f}_1 \cdot \dotso \cdot  \vec{a}_{p_k(n)} \tilde{f}_k \ d\mu_{\tilde{X}} \right)\right|<\varepsilon,\]
	where $\vec{a}_{p_i(n)}=a_1^{p_{1,i}}(n)\cdot \dotso \cdot a_d^{p_{d,i}(n)}$, for all $n \in \Z$.
	\\ \\
	From this point, we proceed exactly as in the remainder of the proof of Theorem \ref{mainthm}, simply changing the polynomial sequence from $(a_1,\dots,a_d)^n$ to $(\vec{a}_{p_1(n)},\dots,\vec{a}_{p_k(n)})$, when applying Theorem 2.5 of \cite{leibman1}. We are done.
\end{proof}
\begin{remark} Using the results for $\Z^d$-actions in \cite{griesmer} and Theorem 0.3 in \cite{leibman2}, one can extend the proofs in Theorem \ref{mainthm} and Theorem \ref{mainthmgen} to polynomials of several variables. The proofs for this setup are essentially the same. 
\end{remark}
\section{Large intersections for two commuting transformations}
The purpose of this section is to prove Theorem \ref{intro2} and Theorem \ref{intro3}. 
\subsection{Large ``linear'' returns}
We begin with the following Theorem, which establishes a limit formula for averages of two commuting transformations with some ergodicity assumptions:
\begin{theorem}\label{twotranslimit}
	Let $(X,\mathcal{B},\mu,T,S)$ be a measure preserving system such that $T, S, TS^{-1}$ are ergodic. Let $Z$ be a compact abelian group such that the Kronecker factor of $X$, denoted by $Z_1(X)$,  is isomorphic to $(Z,\textrm{Borel}(Z),\mu_{\textrm{Haar}},T,S)$. Take $\alpha, \beta \in Z$ so that the map $R_{\alpha}z:=z+\alpha$ on $Z$ corresponds to the action of $T$ on $Z$, and the map $R_{\beta}z:=z+\beta$ corresponds to the action of $S$ on $Z$. Let $Y_{T,S}:=\overline{\{(n\alpha, n\beta) : n \in \Z\}}$, and denote by $\nu_{Y_{T,S}}$ the Haar probability measure on $Y$. Then, for any $f_1, f_2 \in L^{\infty}(\mu)$ and any F\o lner sequence $(F_N) \subseteq \Z$ we have
	\begin{equation}\label{tslimit}
		\lim_{N \to \infty} \frac{1}{|F_N|}\sum_{n \in F_N} T^nf_1S^nf_2=\int_{Y_{T,S}} \tilde{f_1}(z+u)\tilde{f_2}(z+v) \ d\nu_{Y_{T,S}}(u,v) \textrm{ (with respect to the } L^2(\mu)\textrm{-norm)},
	\end{equation}
	where $\tilde{f_i}=\mathbb{E}[f_i|Z_1(X)]$ is the conditional expectation of $f_i$ onto $Z_1(X)$.
\end{theorem}
\begin{proof}
	Let $f_1, f_2 \in L^{\infty}(\mu)$ and $(F_N)$ be a F\o lner sequence in $\Z$. We consider the expression on the left hand side of \eqref{tslimit}. By Proposition \ref{seminorm1}, together with Lemma \ref{lem2}, we have
	\begin{equation}
		\lim_{N \to \infty} \left|\left|\frac{1}{|F_N|}\sum_{n \in F_N} T^nf_1S^nf_2 \right|\right|_{L^2(\mu)} \leq |||f_i|||_{T_i,T_i},
	\end{equation}
	for $i=1,2$, where we put $T_1=T, T_2=S$. In particular, the structure theory afforded by Theorem \ref{frak}
	 it is bounded by $\min \{|||f_1|||_2, |||f_2|||_2\}$ (the seminorms agree, independently of the chosen transformations). Thus, since $|||f_i|||_2=0$ if and only if $\mathbb{E}[f_i|Z_1(X)]=0$, it follows that
	\begin{equation}\label{l2equality1}
		\lim_{N \to \infty} \frac{1}{|F_N|}\sum_{n \in F_N} T^nf_1(x)S^nf_2(x)= \lim_{N \to \infty} \frac{1}{|F_N|}\sum_{n \in F_N} \tilde{f_1}(z+n\alpha)\tilde{f_2}(z+n\beta),
	\end{equation}
	with respect to the $L^2(\mu)$-norm, and where the averages on the right hand side of \eqref{l2equality1} take place in the factor $Z_1(X)$ (in other words, the map $x \mapsto z$ denotes the factor map from $X$ to $Z$). Now, by the ergodic theorem applied to the ergodic action $(z,w) \mapsto (z+\alpha,w+\beta)$ defined on $Y_{T,S}$ we get the following $L^2(\mu)$ limit formula:
	\[ \lim_{N \to \infty} \frac{1}{|F_N|}\sum_{n \in F_N} F(z_1+n\alpha,z_2+n\beta)=\int_{Y_{T,S}} F(u,v) \ d\nu_{Y_{T,S}}(u,v),\]
	for all $(z_1,z_2) \in Z^2$, and all $F \in L^2(Y_{T,S})$. Thus, \eqref{tslimit} follows by setting $F=\tilde{f_1} \otimes \tilde{f_2}$. 
\end{proof}
We will use Theorem \ref{twotranslimit} along with the notation introduced therein to show the next lemma:
\begin{lemma}\label{contfunc}
	Let $(X,\mathcal{B},\mu,T,S)$ be a measure preserving system with $T, S, TS^{-1}$ ergodic. Let $f_0, f_1, f_2 \in L^{\infty}(\mu)$ and $(F_N) \subseteq \Z$ be a F\o lner sequence. Then, for every continuous function $\eta: Y_{T,S} \rightarrow \C$ we have
	\begin{equation}
		\lim_{N \to \infty} \frac{1}{|F_N|}\sum_{n \in F_N} \eta(n\alpha, n\beta)\int_X f_0\cdot T^nf_1 \cdot S^nf_2\ d\mu=\int_Z\int_{Y_{T,S}} \eta(u,v)f_0(z)\tilde{f_1}(z+u)\tilde{f_2}(z+v) \ d \nu_{Y_{T,S}}(u,v) \ d\nu_Z(z),
	\end{equation}
	where $\tilde{f_i}(z)=\mathbb{E}[f_i|Z_1(X)](z)$ is the projection onto the Kronecker factor. 
\end{lemma}
\begin{proof} We begin by observing that since $Y_{T,S}$ is closed, we can extend $\eta$ to a continuous map $\eta_0: Z^2 \rightarrow \C$. Now, by a standard approximation argument using Stone-Weierstrass' theorem, it is enough to show the result for $\eta_0(u,v)=\chi_1(u)\chi_2(v)$, where $\chi_i$ is a character on $Z$. Now put $g_0(x):=\overline{\chi_1(z)\chi_2(z)}f_0(x)$, $g_1(x):=\chi_1(z)f_1(x)$ and $g_2(x):=\chi_2(z)f_2(x)$ with the understanding that we write $z$ for the projection of $x$ onto $Z$.
	\\ \\
	By Theorem \ref{twotranslimit}, we can evaluate the limit of the averages in question:
	\[
	\lim_{N \to \infty} \frac{1}{|F_N|}\sum_{n \in F_N} \eta(n\alpha, n\beta)\int_X f_0\cdot T^nf_1 \cdot S^nf_2\ d\mu=\lim_{N \to \infty} \frac{1}{|F_N|}\sum_{n \in F_N} \eta_0(n\alpha, n\beta)\int_X f_0\cdot T^nf_1 \cdot S^nf_2\ d\mu=\]
	\[\lim_{N \to \infty} \frac{1}{|F_N|}\sum_{n \in F_N} \int_X g_0\cdot T^ng_1 \cdot S^ng_2\ d\mu=
	\int_X g_0(x) \int_{Y_{T,S}} \tilde{g_1}(z+u)\tilde{g_2}(z+v) \ d\nu_{Y_{T,S}}(u,v)\ d\mu(x)=\]
	\[\int_Z \tilde{g_0}(z) \int_{Y_{T,S}} \tilde{g_1}(z+u)\tilde{g_2}(z+v) \ d\nu_{Y_{T,S}}(u,v)\ d\nu_Z(z)=\]
	\[ \int_{Z} \int_{Y_{T,S}} \eta_0(u,v)\tilde{f_0}(z)\tilde{f_1}(z+u)\tilde{f_2}(z+v) \ d\nu_Z(z) \ d\nu_{Y_{T,S}}(u,v)=\]
	\[\int_{Z} \int_{Y_{T,S}} \eta(u,v)\tilde{f_0}(z)\tilde{f_1}(z+u)\tilde{f_2}(z+v) \ d\nu_Z(z) \ d\nu_{Y_{T,S}}(u,v),\]
	as desired.
\end{proof}
With this we can now prove Theorem \ref{intro2}:
\begin{theorem}\label{largelinear}
	Let $(X,\mathcal{B},\mu,T,S)$ be a measure preserving system such that $T, S, TS^{-1}$ are ergodic. Let $\varepsilon>0$ and $A \in \mathcal{B}$. Then, the set
	\begin{equation}\label{setsynd}
		\{ n \in \Z : \mu(A \cap T^{-n}A \cap S^{-n}A)>\mu(A)^3-\varepsilon \}
	\end{equation}
	is syndetic.
\end{theorem}
\begin{proof} We proceed by contradiction, so assume that the set \eqref{setsynd} is not syndetic. Then there exists $\varepsilon>0$, $A \in \mathcal{B}$ with $\mu(A)>0$ and we can find a F\o lner sequence $(F_N)$ such that
	\begin{equation}\label{ineq0}
		\mu(A \cap T^{-n}A \cap S^{-n}A) \leq \mu(A)^3-\varepsilon
	\end{equation}
	for all $n \in \bigcup_{N \in \N} F_N$. Let $f=\mathbb{1}_A$ and write $\tilde{f}$, for the projection of $f$ onto the Kronecker factor of $X$. We know that $0 \leq \tilde{f} \leq 1$, so by Jensen's inequality we have
	\begin{equation}
		\int_Z \tilde{f}(z)\tilde{f}(z)\tilde{f}(z) \ d\nu_Z(z) \geq \left( \int_Z \tilde{f}(z) \ d\nu_Z(z)\right)^3=\mu(A)^3.
	\end{equation}
	Consequently, for any $(u,v)$ in a small neighborhood of the identity of $Z^2$, we have that
	\begin{equation}
		\int_Z \tilde{f}(z)\tilde{f}(z+u)\tilde{f}(z+v) \ d\nu_Z(z) \geq \mu(A)^3-\frac{\varepsilon}{2}.
	\end{equation}
	By Urysohn's lemma we can find a continuous function $\eta : Y_{T,S} \rightarrow [0,\infty)$ such that $\int_{Y_{T,S}} \eta \ d\nu_{Y_{T,S}}=1$ satisfying
	\begin{equation}\label{ineq1}
		\int_{Y_{T,S}}\int_Z \eta(u,v)\tilde{f}(z)\tilde{f}(z+u)\tilde{f}(z+v) \ d\nu_Z(z) \ d\nu_{Y_{T,S}}(u,v) \geq \mu(A)^3-\frac{\varepsilon}{2}.
	\end{equation}
	Now, Lemma \ref{contfunc} and \eqref{ineq1} give
	\begin{equation}
		\lim_{N \to \infty} \frac{1}{|F_N|}\sum_{n \in F_N} \eta(n\alpha, n\beta) \mu(A \cap T^{-n}A \cap S^{-n}A) \geq \mu(A)^3-\frac{\varepsilon}{2}.
	\end{equation}
	But \eqref{ineq0} implies that 
	\[ \limsup_{N \to \infty} \frac{1}{|F_N|} \sum_{n \in F_N} \eta(n\alpha, n\beta) \mu(A \cap T^{-n}A \cap S^{-n}A) \leq (\mu(A)^3-\varepsilon) \limsup_{N \to \infty} \frac{1}{|F_N|}\sum_{n \in F_N} \eta(n\alpha,n\beta)=\] \[(\mu^3(A)-\varepsilon)\int_{Y_{T,S}} \eta \ d\nu_{Y_{T,S}}=\mu(A)^3-\varepsilon,\]
	a contradiction.
\end{proof}
\subsection{Large returns along shifted primes}
In this subsection we will establish Theorem \ref{intro3}. We begin by introducing the \emph{von Mangoldt function} $\Lambda : \Z \rightarrow \R$, given by
\[ \Lambda(n):=\begin{cases}
\log p, \textrm{ if } n=p^m, \textrm{ for some } m \in \N, p \in \mathbb{P} \\
0 \textrm{ otherwise.}
\end{cases} \]
Let $w \in \N$ and $r \in \Z$. Put $W:=\prod_{p<w, p \in \mathbb{P}} p$ and $\Lambda'(n):=\Lambda(n)\mathbb{1}_{\mathbb{P}}(n)$. Then, for $n \in \N$ the modified von Mangoldt function is defined by
\[ \Lambda'_{w,r}(n):=\frac{\phi(W)}{W}\Lambda'(Wn+r),\]
where $\phi$ is the Euler totient function. 
We begin with a classical lemma (cf. Lemma 2.1 in \cite{fhk}):
\begin{lemma}\label{compare}
	Let $(a_n)_{n \in \N}$ be a sequence of complex numbers with $|a_n| \leq 1$ for all $n \in \N$. Then,
	\begin{equation}
		\left| \frac{1}{\pi(N)}\sum_{p \in \mathbb{P},p<N} a_p-\frac{1}{N}\sum_{n=0}^{N-1} \Lambda'(n)a_n \right|=o_{N \to \infty}(1),
	\end{equation}
	where $\pi(N)$ is the number of primes less than or equal to $N$.
\end{lemma}
The next result is taken from \cite{fhk}. We put .
\begin{proposition}[cf. Proposition 3.6 in \cite{fhk}]\label{removinglambda} Let $(X,\mathcal{B},\mu,T,S)$ be a measure preserving system. Let $f, g \in L^{\infty}(\mu)$. Then,
	\[ \lim_{w \to \infty} \lim_{N \to \infty}\max_{(r,W)=1, r \leq W}\left|\left|\frac{1}{N}\sum_{n=1}^N (\Lambda_{w,r}'(n)-1)T^nfS^ng\right|\right|_{L^2(\mu)}=0.\]
	In particular, from the proof of Proposition 3.6 in \cite{fhk}, we see that letting $B_{W,r}(N):=\frac{1}{N}\sum_{n=1}^N T^{Wn+r}fS^{Wn+r}g$ we have
	\begin{equation}
		\max_{r<W, (r,W)=1} \left|\left| \frac{1}{N}\sum_{n=1}^N \Lambda'_{w	,r}(n)T^{Wn+r}fS^{Wn+r}g-B_{W,r}(n) \right|\right|_{L^{2}(\mu)}=o_{N\to \infty,w}(1)+o_{w \to \infty}(1).
	\end{equation}
\end{proposition}
In particular, this implies that the Kronecker factor is characteristic in the following setting:
\begin{lemma}[Adapted from \cite{le}]\label{commutingle}
	Let $(X,\mathcal{B},\mu,T,S)$ be a measure preserving system with $T, S, TS^{-1}$ totally ergodic. Let $f, g \in L^{\infty}(\mu)$. Then, if $|||f|||_2=0$ or if $|||g|||_2=0$, then 
	\begin{equation}
		\lim_{N \to \infty}\frac{1}{N}\sum_{n=1}^N T^{p_n-1}fS^{p_n-1}g=0 \textrm{ in } L^2(\mu).
	\end{equation}
\end{lemma}
\begin{proof}
	First notice that letting $a(n):=T^nfS^ng$ it suffices to show that $\frac{1}{N}\sum_{n=1}^Na(p_n)$ converges to $0$ in $L^2(\mu)$ as $N \to \infty$ if $|||f|||_2=0$ or $|||g|||_2=0$. Now, by Proposition \ref{removinglambda}
	\begin{equation}\label{eq1}
		\left|\left| \frac{1}{W}\sum_{1 \leq r \leq W, (r,W)=1}\frac{1}{N}\sum_{n=1}^N \Lambda'_{W,r}(n)a(Wn+r)-\frac{1}{W}\sum_{r=1, (r,W)=1}^W B_{W,r}(n)\right|\right|_{L^2(\mu)}=o_{N \to \infty,w}(1)+o_{w \to \infty}(1).
	\end{equation}
	But the first average in \eqref{eq1} is equal to $\frac{1}{WN}\sum_{n=1}^{WN}\Lambda'_{W,r}(n)a(n)$. It follows from Lemma \ref{contfunc}, that the Kronecker factor is characteristic for the averages $B_{W,r}(N)=\frac{1}{N}\sum_{n=1}^NT^{Wn+r}fS^{Wn+r}g$, and thus, $\lim_{N \to \infty} \frac{1}{N}\sum_{n=1}^N B_{W,r}(n)=0$ in $L^2(\mu)$.
	\\ \\
	It was shown in \cite{fhk} that the limit $\lim_{N \to \infty} \frac{1}{N}\sum_{n=1}^N\Lambda'(n)a(n)$ exists in $L^2(\mu)$. Equation \eqref{eq1} implies that it must be equal to $0$, so we are done.
\end{proof}
Proposition \ref{removinglambda} and Lemma \ref{commutingle} allow us to generalize Theorem 1.1 in \cite{le} and obtain the following theorem:
\begin{theorem}\label{legeneral}
	Let $(X,\mathcal{B},\mu,T,S)$ be a measure preserving system with $T, S, TS^{-1}$ totally ergodic. Let $f_0,f_1,f_2 \in L^{\infty}(\mu)$ and consider the multicorrelation sequence
	\[ a(n):=\int_X f_0 \cdot T^nf_1 \cdot S^nf_2 \ d\mu.\]
	Let $a(n)=a_{\textrm{st}}(n)+a_{\textrm{er}}(n)$ be the decomposition of the multicorrelation sequence $a(n)$ obtained in Theorem \ref{mainthm}. Then,
	\[ \lim_{N-M \to \infty} \frac{1}{N-M}\sum_{n=M}^{N-1} |a_{\textrm{er}}(p_n-1)|=0.\]
\end{theorem}
\begin{proof}
	We start by noticing that Proposition \ref{removinglambda} and Lemma \ref{commutingle} imply that the Kronecker factor (which is a special case of a nilfactor) is characteristic for the $L^2(\mu)$-limit of the averages
	\[ \lim_{N \to \infty} \frac{1}{N}\sum_{n=1}^N T^{p_n-1}f_1S^{p_n-1}f_2\]
	for all $f_1, f_2 \in L^{\infty}(\mu)$. Combining this with the equidistribution results obtained in \cite{le} (Corollary 1.4), we see that for commuting $a_1, a_2$ in the Kronecker factor $Z$, $(a_1^{p_n-1},a_2^{p_n-1})$ is equidistributed with appropriate weights on the connected components of $Z$. Using uniqueness of the decomposition into a nilsequence and a nullsequence of the multicorrelation function $a(n)$, it follows that putting
	\[ a(n)=a_{\textrm{st}}(n)+a_{\textrm{er}}(n)\]
	(as obtained in Theorem \ref{mainthm}), the sequence $a_{\textrm{er}}$ is still a nullsequence along primes, namely 
	\[ \lim_{N \to \infty} \frac{1}{N}\sum_{n=1}^N |a_{\textrm{er}}(p_n-1)|=0.\]
\end{proof}
We are now in a position to prove Theorem \ref{intro3}:
\begin{theorem}\label{cubicprime}
	Let $(X,\mathcal{B},\mu,T,S)$ be a measure preserving system with $T, S, TS^{-1}$ totally ergodic. Then, for all $\varepsilon>0$ and all $A \in \mathcal{B}$, the set 
	\[ \{n \in \Z : \mu(A \cap T^{-(p_n-1)}A \cap S^{-(p_n-1)} A)>\mu(A)^3-\varepsilon\} \]
	has positive lower density. 
\end{theorem}
\begin{proof} We follow the method described in \cite{dmls}. Let $\phi(n):=\mu(A \cap T^{-n}A\cap S^{-n}A)$, then by Theorem \ref{mainthm}, $\phi(n)=a_{\textrm{st}}(n)+a_{\textrm{er}}(n)$, where $a_{\textrm{st}}(n)$ is a uniform limit of $2$-step nilsequences and, by Theorem \ref{legeneral}, we have
	\[ \lim_{N \to \infty} \frac{1}{N}\sum_{n=1}^N |a_{\textrm{er}}(n)|=\lim_{N \to \infty}\frac{1}{N}\sum_{n=1}^N|a_{\textrm{er}}(p_n-1)|=0.\]
	Let $\varepsilon>0$. Since $a_{\textrm{st}}(n)$ is a uniform limit of $2$-step nilsequences it can be approximated by a nilsequence $F(\tau^n\Gamma)$, where $Y=G/\Gamma$ is a $2$-step nilmanifold, $F \in C(Y)$, and $\tau$ acts ergodically on $Y$, which we assume has $d$ connected components. Moreover, we can assume that $|F(\tau^n\Gamma)-a_{\textrm{st}}(n)|<\varepsilon/4$ for all $n \in \N$. We further suppose that $\Gamma \in Y$.
	\\ \\
	Note that in Theorem \ref{largelinear} if one strengthens the assumptions to $T, S, TS^{-1}$ being totally ergodic, then the set $S_d:=\{ n \in \N : \mu(A \cap T^{-dn}A \cap S^{-dn} A)=\phi(dn)>\mu(A)^3-\varepsilon\}$ is syndetic. Thus,
	\begin{equation}
		\lim_{N \to \infty} \frac{1}{|\{1,\dots,N\} \cap S_d|}\sum_{0\leq n<N, n \in S_d} |a_{\textrm{er}}(dn)|=0, \textrm{ whence}
	\end{equation}
	\begin{equation}
		\limsup_{N \to \infty} \frac{1}{|\{1,\dots,N\} \cap S_d|}\sum_{0\leq n<N, n \in S_d}|\phi(dn)-F(\tau^{dn}\Gamma)|<\frac{\varepsilon}{4}.
	\end{equation}
	This, in turn, implies that there is some $n \in \N$ for which $F(\tau^{dn}\Gamma)>\mu^3(A)-\frac{\varepsilon}{2}$. Since $\tau^{dn}\Gamma \in Y_0$, we can find an open set $U$ of $Y_0$ such that $F>\mu^3(A)-\frac{3\varepsilon}{4}$ on $U$. By Corollary 1.4 in \cite{le}, the sequence $(\tau^{p_n-1}\Gamma)$ is equidistributed on $Y_0$ when restricted to $p_n=1 \pmod{d}$. Hence, the set $R:=\{ n \in \N : \tau^{p_n-1}\Gamma \in U\}$ has positive lower density, and for every $n \in R$ we have $F(\tau^{p_n-1}\Gamma)>\mu^3(A)-\frac{3\varepsilon}{4}$. Moreover, the set of $R':=\{n \in \N : \phi(p_n-1)<\mu^3(A)-\varepsilon\}$ has $0$ density. The set $R \setminus R'$ gives the desired result.
\end{proof}
\bibliographystyle{alpha}

\end{document}